\documentclass[a4paper]{article}
\usepackage[english]{babel}
\usepackage[utf8]{inputenc}
\usepackage{amsmath}
\usepackage{graphicx}
\usepackage[colorinlistoftodos]{todonotes}

\usepackage{amsthm}
\usepackage{amsfonts}
\usepackage{amssymb}
\usepackage{amscd}
\usepackage{stmaryrd}
\usepackage{float}
\usepackage[all]{xy}
\usepackage{geometry}
\usepackage{verbatim}

\newcommand{\A}{\ensuremath \mathcal{A}}
\newcommand{\B}{\ensuremath \mathcal{B}}

\newcommand{\Q}{\ensuremath \mathbb{Q}}

\newcommand{\C}{\ensuremath \mathbb{C}}
\newcommand{\Z}{\ensuremath \mathbb{Z}}

\newcommand{\isom}{\ensuremath \cong}

\theoremstyle{plain}
\newtheorem{theorem}{Theorem}[section]
\newtheorem{proposition}[theorem]{Proposition}
\newtheorem{corollary}[theorem]{Corollary}
\newtheorem{lemma}[theorem]{Lemma}

\theoremstyle{definition}
\newtheorem{definition}[theorem]{Definition}
\newtheorem{example}[theorem]{Example}

\newcommand{\thmref}[1]{Theorem \ref{#1}}
\newcommand{\corref}[1]{Corollary \ref{#1}}
\newcommand{\lemref}[1]{Lemma \ref{#1}}
\newcommand{\propref}[1]{Proposition \ref{#1}}

\newcommand{\defref}[1]{Definition \ref{#1}}
\newcommand{\exampleref}[1]{Example \ref{#1}}

\title{An Isomorphism Extension Theorem For \\ Landau-Ginzburg B-Models}

\author{Nathan Cordner}

\date{July 1, 2016}

\begin{document}

\newgeometry{top = 1 in, left = 1 in, right = 1 in, bottom = 1 in}

\maketitle

\begin{abstract}
Landau-Ginzburg mirror symmetry studies isomorphisms between A- and B-models, which are graded Frobenius algebras that are constructed using a weighted homogeneous polynomial $W$ and a related group of symmetries $G$ of $W$. It is known that given two polynomials $W_{1}$, $W_{2}$ with the same weights and same group $G$, the corresponding A-models built with ($W_{1}$,$G$) and ($W_{2}$,$G$) are isomorphic. Though the same result cannot hold in full generality for B-models, which correspond to orbifolded Milnor rings, we provide a partial analogue. In particular, we exhibit conditions where isomorphisms between unorbifolded B-models (or Milnor rings) can extend to isomorphisms between their corresponding orbifolded B-models (or orbifolded Milnor rings). 
\end{abstract}

\section{Introduction}

Landau-Ginzburg mirror symmetry studies two different physical theories, known as Landau-Ginzburg A- and B-models, which are graded Frobenius algebras that are built using a nondegenerate weighted homogeneous polynomial $W$ and a related group of symmetries $G$ of $W$. The A-model theories (denoted by $\A$) have been constructed \cite{FJR07}, and are a special case of what is known as \emph{FJRW theory}. The B-model theories (denoted by $\B$) have also been constructed \cite{IV,Ka1,Ka2,Ka3,Kra09}, and correspond to an \emph{orbifolded Milnor ring}. In many cases, these theories extend to whole families of Frobenius algebras, called \emph{Frobeinus manifolds}.

For a large class of polynomials, Berglund-H\"ubsch \cite{BH}, Henningson \cite{Hen}, and Krawitz \cite{Kra09} described the construction of a dual (or transpose) polynomial $W^{T}$ and a dual group $G^{T}$. The Landau-Ginzburg mirror symmetry conjecture states that the A-model of a pair $(W,G)$ should be isomorphic to the B-model of the dual pair $(W^{T}, G^{T})$, and is denoted as $\A[W,G] \isom \B[W^{T}, G^{T}]$. This conjecture has been proven in many cases \cite{FJJS11, Kra09}, although the proof of the full conjecture remains open. To better understand mirror symmetry, it has been fruitful to focus on studying isomorphisms between Landau-Ginzburg models of the same type:  either from A to A, or from B to B.

The Landau-Ginzburg A-model is \emph{deformation invariant} (see \cite{Tay13}). Given two polynomials with the same weights, and an admissible symmetry group that fixes both polynomials, there exists a continuous path to deform one polynomial to the next. All the corresponding A-models along such a path are isomorphic as graded Frobenius algebras---this result is sometimes called the \emph{Group-Weights Theorem}. The same result does not hold for B-models (see \exampleref{bmodelcondefcounterex}).

The unorbifolded Landau-Ginzburg B-model, which is built using the trivial group $G$, corresponds to the Milnor ring (or local algebra) of a polynomial $W$ and is often denoted as $\mathcal{Q}_{W}$. The original construction of the vector space structure of the orbifolded Milnor ring, or orbifolded B-model, was given by Intriligator and Vafa \cite{IV}. The product structure remained undefined for many years. Recently, Krawitz \cite{Kra09} followed ideas presented by Kaufmann \cite{Ka1,Ka2,Ka3} to write down a multiplication for the orbifolded Milnor ring.

Classical singularity theory has widely studied Milnor rings of polynomials and their related isomorphisms. In this paper, we look at providing a partial Group-Weights result for orbifolded Milnor rings. That is, we look at extending known isomorphisms between Milnor rings to isomorphisms between orbifolded Milnor rings that have the product structure defined by Krawitz \cite{Kra09}. We will often refer to these orbifolded Milnor rings as Landau-Ginzburg B-models.

We approach this problem by focusing on special choices of polynomials and groups. Building on ideas presented in \cite{FJJS11}, we arrive at the following conditions for a polynomial/group pair. 

\begin{definition}\label{wellbehaved}
A pair $(W,G)$ is \emph{well behaved} if $W = \sum W_{i}$, where each $W_{i}$ is an admissible polynomial in distinct variables, and $G = \bigoplus G_{i}$, where each $g \in G_{i}$ either fixes all or none of the variables of $W_{i}$ for each $i$.
\end{definition}

As we will note later, a large class of polynomial/group pairs that satisfy this condition include the two-variable \emph{admissible} polynomials together with any of their symmetry groups. However, some polynomials in three or more variables (such as \emph{chain} polynomials) may have choices of symmetry groups that do not form well-behaved pairs.

We note that \defref{wellbehaved} is similar to Property (*) of \cite{FJJS11} (see also \defref{propertystar}). We also require that the particular isomorphism between Milnor rings be \emph{equivariant}. That is, when applying a nontrivial group of symmetries, the isomorphism respects the group action on the Milnor ring's vector space basis. The following theorem is the main result of the paper. 

\vspace{0.1 in}

\noindent \textbf{\thmref{mainresult}.} \emph{Let $(W,G)$ and $(V, G)$ be well behaved. If $\phi: \mathcal{Q}_{W} \rightarrow \mathcal{Q}_{V}$ is an equivariant isomorphism of graded Frobenius algebras, then $\phi$ extends to an isomorphism $\psi: \B[W, G] \rightarrow \B[V, G]$.}

\vspace{0.1 in}

It turns out that when Milnor rings are isomorphic, the corresponding polynomials will have the same weights (up to ordering of variables, see \thmref{isommilnorringweights}). Therefore \thmref{mainresult} is a partial Group-Weights theorem for orbifolded Milnor rings. In Section 4 we give examples of cases where this theorem applies.

\section{Preliminaries}

Here we will introduce some of the concepts needed to understand the theory of this paper. 

\subsection{Admissible Polynomials}

\begin{definition}
For a polynomial $W \in \C[x_{1},\dots,x_{n}]$, we say that $W$ is \textit{nondegenerate} if it has an isolated critical point at the origin.
\end{definition}

\begin{definition}
Let $W \in \C[x_{1},\dots,x_{n}]$. We say that $W$ is \textit{quasihomogeneous} if there exist positive rational numbers $q_{1},\dots,q_{n}$ such that for any $c \in \C$, $W(c^{q_{1}}x_{1},\dots,c^{q_{n}}x_{n}) = c W(x_{1},\dots,x_{n})$.
\end{definition}

We often refer to the $q_{i}$ as the \textit{quasihomogeneous weights} of a polynomial $W$, or just simply the \textit{weights} of $W$, and we write the weights in vector form $J = (q_{1}, \dots, q_{n})$. 

\begin{definition}
$W \in \C[x_{1},\dots,x_{n}]$ is \emph{admissible} if $W$ is nondegenerate and quasihomogeneous with unique weights, having no monomials of the form $x_{i}x_{j}$ for $i \ne j$.
\end{definition}

The condition that $W$ have no cross-term monomials is necessary for constructing the A-model. Because the construction of $\A[W,G]$ requires an admissible polynomial, we will only be concerned with admissible polynomials in this paper. In order for a polynomial to be admissible, it needs to have at least as many monomials as variables. Otherwise its quasihomogeneous weights cannot be uniquely determined. We now state the main subdivision of the admissible polynomials.

\begin{definition}
Let $W$ be an admissible polynomial. We say that $W$ is \textit{invertible} if it has the same number of monomials as variables. If $W$ has more monomials than variables, then it is \textit{noninvertible}.
\end{definition}

Admissible polynomials with the same number of variables as monomials are called invertible, since their associated exponent matrices are square and invertible. 

\begin{definition}
Let $W \in \C[x_{1},\dots, x_{n}]$. If we write $W$ as a sum of monomials $W = \sum_{i = 1}^{m} c_{i} \prod_{j = 1}^{n} x_{j}^{a_{ij}}$, then the associated \textit{exponent matrix} is defined to be $A = (a_{ij})$. 
\end{definition}

We further observe that if $W$ is invertible, we can rescale variables to force each coefficient $c_{i}$ to equal one---which we will do in this paper. The invertible polynomials can also be decomposed into sums of three types of polynomials, called the \emph{atomic types}.

\begin{theorem}[Theorem 1 of \cite{KS}]
Any invertible polynomial is the decoupled sum of polynomials in one of three atomic types:
\begin{align*}
\begin{array}{rl}\text{Fermat type: } & W = x^{a}, \\ 
\text{Loop type: } & W = x_{1}^{a_{1}}x_{2} + x_{2}^{a_{2}}x_{3} + \dots + x_{n}^{a_{n}}x_{1}, \\
\text{Chain type: } & W = x_{1}^{a_{1}}x_{2} + x_{2}^{a_{2}}x_{3} + \dots + x_{n}^{a_{n}}.
\end{array}
\end{align*}
\end{theorem}
We also assume that the $a_{i} \ge 2$ to avoid terms of the form $x_{i}x_{j}$ for $i \ne j$. 

\subsection{Symmetry Groups}

\begin{definition}
Let $W$ be an admissible polynomial. We define the \textit{maximal diagonal symmetry group} of $W$ to be $G_{W}^{\max} = \{(\zeta_{1}, \dots, \zeta_{n}) \in (\C^\times)^{n} \mid W(\zeta_{1}x_{1},\dots,\zeta_{n}x_{n}) = W(x_{1},\dots,x_{n})    \}$.
\end{definition}

The proofs of Lemma 2.1.8 in \cite{FJR07} and Lemma 1 in \cite{ABS11} show that $G_{W}^{\max}$ is finite and that each coordinate of every group element is a root of unity. The group operation in $G_{W}^{\max}$ is coordinate-wise multiplication. But since additive notation is often more convenient, we use the map $(e^{2\pi i \theta_{1}},\dots,e^{2\pi i \theta_{n}}) \mapsto (\theta_{1}, \dots, \theta_{n}) \mod \Z$ taking $G_{W}^{\max}$ to $(\Q / \Z)^{n}$. Hence we will often write $G_{W}^{\max} = \{ g \in (\Q/\Z)^{n} \mid Ag \in \Z^{m} \}$, where $A$ is the $m \times n$ exponent matrix of $W$.

\begin{definition}
In this notation, $G_{W}^{\max}$ is a subgroup of $(\Q/\Z)^{n}$ with respect to coordinate-wise addition. For $g \in G_{W}^{\max}$, we can write $g$ uniquely as $(g_{1}, \dots, g_{n})$, where each $g_{i}$ is a rational number in the interval [0,1). The $g_{i}$ are called the \textit{phases} of $g$.
\end{definition}

That being said, as a matter of convenience we will often use equivalent representatives of the $g_{i}$ that lie outside the interval [0,1) to write down group elements.

\subsection{Graded Frobenius Algebras}

Landau-Ginzburg A- and B-models are algebraic objects that are endowed with many levels of structure. In this paper, we will chiefly be concerned with their structure up to the level of graded Frobenius algebras. We will only develop the theory needed for this paper. We refer the interested reader to \cite{FJR07} for more details on the construction of the A-model. \cite{FJJS11}, \cite{Kra09}, and \cite{Tay13} also contain more information on constructing A- and B-models, and related isomorphisms. 

\begin{definition}
A \emph{graded Frobenius algebra}  is a graded algebra $A$ with a \emph{pairing} $\langle \cdot, \cdot \rangle: A \times A \rightarrow \C$ that is 

$\bullet$ Symmetric:  $\langle x, y \rangle = \langle y, x \rangle$ for all $x,y \in A$,

$\bullet$ Linear:  $\langle \alpha x + \beta y, z \rangle = \alpha \langle x, z \rangle + \beta \langle y, z \rangle$ for all $x,y,z \in A$ and $\alpha, \beta \in \C$,

$\bullet$ Nondegenerate:  for every $x \in A$ there exists $y \in A$ such that $\langle x, y \rangle \ne 0$. 

\noindent The pairing further satisfies the \textit{Frobenius property}, meaning that $\langle x \cdot y, z \rangle = \langle x, y\cdot z \rangle$ for all $x,y,z \in A$.
\end{definition}

\subsection{Unorbifolded B-Models}

\begin{definition}
For any polynomial $W$, the algebra $\mathcal{Q}_{W} = \C[x_{1},\dots,x_{n}] / (\frac{\partial W}{\partial x_{1}}, \dots, \frac{\partial W}{\partial x_{n}})$ is called the \textit{Milnor ring} (or \textit{local algebra}) of $W$.
\end{definition}

We note that $\mathcal{Q}_{W}$ has a vector space structure with a basis consisting of monomials.

\begin{theorem}[Theorem 2.6 of \cite{Tay13}]
If $W$ is admissible, then $\mathcal{Q}_{W}$ is finite dimensional.
\end{theorem}

We will further think of the Milnor ring as a graded vector space over $\C$, by defining the degree of a monomial in $\mathcal{Q}_{W}$ to be $\deg(x_{1}^{a_{1}}x_{2}^{a_{2}}\dots x_{n}^{a_{n}}) = 2\sum_{i}^{n} a_{i}q_{i}$, where the $q_{i}$ are the quasihomogeneous weights of $W$. We have the following well-known results about the vector space structure of the Milnor ring (see Section 2.1 of \cite{Kra09}). First, $\dim(\mathcal{Q}_{W}) = \prod_{i = 1}^{n} \left( \frac{1}{q_{i}} - 1 \right)$. Second, the highest degree of its graded pieces is $2\sum_{i = 1}^{n} \left(1 - 2q_{i} \right)$. The number $\sum_{i = 1}^{n} \left(1 - 2q_{i} \right)$ is called the \emph{central charge}, and is denoted by $\widehat{c}$.

To make $\mathcal{Q}_{W}$ into a graded Frobenius algebra, we need to define a pairing function. 

\begin{definition}
For an admissible polynomial $W$, let $m,n \in \mathcal{Q}_{W}$. We define the \emph{pairing} $\langle m, n \rangle$ to be the complex number that satisfies
\begin{align*}
mn = \frac{\langle m, n \rangle}{\mu} \text{Hess}(W) + \text{ terms of degree less than} \deg(\text{Hess}(W)), 
\end{align*}
where $\mu$ is the dimension of $\mathcal{Q}_{W}$ as a vector space and Hess$(W)$ is the \emph{Hessian} of $W$---or the determinant of the matrix of second partial derivatives of $W$.  
\end{definition}

As noted by Krawitz \cite{Kra09}, we can represent Hess$(W)$ as a monomial in the Milnor ring. Further, the elements of highest degree in the Milnor ring form a one-dimensional subspace that is spanned by Hess$(W)$. 

The Milnor ring, together with the grading of the monomial basis and this pairing function, forms a graded Frobenius algebra. 

\begin{definition}
We define the \textit{unorbifolded} B-model $\B[W, \{\mathbf{0}\}]$ by $\B[W, \{\mathbf{0}\}] = \mathcal{Q}_{W}$.
\end{definition}

\subsection{Orbifolded B-Models}


Constructing orbifolded B-models for a general group $G$ has historically been a hard problem. Kaufmann did a lot of work in this area \cite{Ka1,Ka2, Ka3}, but in this paper we focus on the most important case for Landau-Ginzburg mirror symmetry:  the diagonal subgroup of SL$(n,\C)$. We follow the construction of Krawitz \cite{Kra09}, built on the work of Kaufmann.

\begin{definition}
Let $W \in \C[x_{1},\dots,x_{n}]$ be admissible, and let $g = (g_{1}, \dots, g_{n}) \in G_{W}^{\max}$. The \textit{fixed locus} of the group element $g$ is the set fix$(g) = \{x \in \C^{n} \mid g(x) = 0  \}$.
\end{definition}

The notation $W|_{\text{fix}(g)}$ denotes the restriction of the polynomial $W$ to the variables in the set fix$(g)$. We now state how $G$ acts on the Milnor ring.

\begin{definition}\label{grp_action}
Let $W$ be an admissible polynomial, and let $g \in G_{W}^{\max}$. We define the map $g^{*}: \mathcal{Q}_{W} \rightarrow \mathcal{Q}_{W}$ by $g^{*}(m) = \det(g) m \circ g$. (Here we think of $g$ as being a diagonal map with multiplicative coordinates). This is the \emph{group action} on the elements of $\mathcal{Q}_{W}$, sometimes denoted as $g \cdot m$. 
\end{definition}

\begin{definition}
Let $W$ be an admissible polynomial, and let $G \le G_{W}^{\max}$. The \emph{$G$-invariant subspace} of $\mathcal{Q}_{W}$ is defined to be $\mathcal{Q}_{W}^{G} = \{m \in \mathcal{Q}_{W} \mid g^{*}(m) = m \text{ for each } g \in G \}$.
\end{definition}

To construct an orbifolded B-model, we restrict $G$ to be a subgroup of $G_{W}^{\max} \cap \text{SL}(n, \C)$.

\begin{definition}
Let $W$ be an admissible polynomial, and $G \le G_{W}^{\max} \cap \text{SL}(n,\C)$ where $n$ is the number of variables of $W$. We define the underlying vector space of $\B[W,G]$ to be $\bigoplus\limits_{g \in G}\left( \mathcal{Q}_{W|_{\text{fix}(g)}} \right)^{G}$, where $( \cdot )^{G}$ denotes all the $G$-invariants. This is called the B-model \textit{state space}.
\end{definition}

The condition that $G \le G_{W}^{\max} \cap \text{SL}(n,\C)$ is required to construct the orbifolded B-model. We will often denote the group $G_{W}^{\max} \cap \text{SL}(n,\C)$ as SL$(W)$.

Note that if we let $G = \{\mathbf{0}\}$, then the formula yields the Milnor ring of $W$, as expected. We also note that the vector space basis of $\B[W,G]$ is made up of monomials from the basis of $\mathcal{Q}_{W|_{\text{fix}(g)}}$ for each $g \in G$. We denote these basis elements $\lfloor m; g \rceil$, where $g$ is a group element and $m$ is a monomial in $\left( \mathcal{Q}_{W|_{\text{fix}(g)}} \right)^{G}$. 

To make $\B[W,G]$ into a graded Frobenius algebra, we will define the grading, the multiplication and the pairing function. We'll start with the vector space grading.

\begin{definition}
Let $W$ be an admissible polynomial with weights $(q_{1}, \dots, q_{n})$. For a basis element $\lfloor m; (g_{1}, \dots, g_{n}) \rceil$ in the vector space basis for $\B[W, G]$, we define its \emph{degree} to be
\begin{align*}
2p + \sum_{g_{i} \notin \Z} (1 - 2q_{i}),
\end{align*}
where $p$ is the weighted degree of $m$. That is, if $m = x_{1}^{a_{1}} \cdots x_{n}^{a_{n}}$, then $p = \sum_{i = 1}^{n} a_{i}q_{i}$.
\end{definition}

The definition of B-model multiplication is due to Krawitz \cite{Kra09}, Kaufmann \cite{Ka1,Ka2,Ka3}, and Intriligator-Vafa \cite{IV}.

\begin{definition}\label{bmodelmultdef}
The product of two elements $\lfloor m; g\rceil$ and $\lfloor n ; h \rceil$ is given by
\begin{align*}
\lfloor m; g \rceil \star \lfloor n; h \rceil = \left\{\begin{array}{ll} \lfloor \gamma nm; g + h \rceil & \text{if fix}(g) \cup \text{ fix}(h) \cup \text{ fix}(g + h) = \C^{n} \\ 0 & \text{otherwise}   \end{array} \right. 
\end{align*}
where $\gamma$ is a monomial defined as
\begin{align*}
\gamma = \frac{\mu_{g \cap h}\text{Hess}(W|_{\text{fix}(g+h)} )}{\mu_{g + h}\text{Hess}(W|_{\text{fix}(g) \cap \text{fix}(h)}) }.
\end{align*}
Here $\mu_{g \cap h}$ is the dimension of the Milnor ring of $W|_{\text{fix}(g) \cap \text{fix}(h)}$, and $\mu_{g + h}$ is the dimension of the Milnor ring of $W|_{\text{fix}(g +h)}$.
\end{definition}

We note that Krawitz proved this multiplication to be associative in the case that $W$ is an invertible polynomial (see Proposition 2.1 of \cite{Kra09}). We believe this to also always be associative when $W$ is noninvertible polynomial, but it has never been proven in general. 

Finally, we have the pairing function.

\begin{definition}
Let $\lfloor m; g \rceil$ and $\lfloor n; h \rceil$ be two basis elements of $\B[W, G]$. If $g = -h$, then $\mathcal{Q}_{W|_{\text{fix}(g)}}$ is canonically isomorphic to $\mathcal{Q}_{W|_{\text{fix}(h)}}$. Therefore, we can define the \emph{pairing} on $B[W,G]$ as follows:
\begin{align*}
\langle \lfloor m; g \rceil, \lfloor n; h \rceil \rangle = \left\{ \begin{array}{ll} \langle m , n \rangle_{\mathcal{Q}_{W|_{\text{fix}(g)}}} & \text{if } g = - h, \\ 0 & \text{otherwise.}    \end{array}   \right.
\end{align*}
\end{definition}

One can verify that the orbifolded B-model $\B[W,G]$, as it has been defined, is a graded Frobenius algebra.

\subsection{Isomorphisms of Graded Frobenius Algebras}

We will focus on studying isomorphisms between Landau-Ginzburg B-models. The following are some common results about isomorphisms between unorbifolded B-models. We will refer back to these later on in the paper. Note that we consider two polynomials to be equivalent if they define the same singularity at the origin. That is, we say that $f \sim g$ if there exists a diffeomorphism $h: \C^{n} \rightarrow \C^{n}$ such that $f = g \circ h$. 

\begin{theorem}[Theorem 2.2.8 of \cite{Suggs}]
If $W_{1}$ and $W_{2}$ are quasihomogeneous functions fixing the origin, then $W_{1}$ and $W_{2}$ are equivalent if and only if their Milnor rings are isomorphic. 
\end{theorem}

\begin{theorem}[Theorem 5.1.1 of \cite{Suggs}]\label{isommilnorringweights}
If two nondegenerate quasihomogeneous polynomials are equivalent, they have the same unordered set of weights. 
\end{theorem}

\begin{theorem}[Webb's Theorem, Theorem 5.1.3 of \cite{Suggs}]\label{webbs_thm}
Let $W_{1}$ and $W_{2}$ be nondegenerate quasihomogeneous polynomials with the same (ordered) weights. If no elements in $\mathcal{Q}_{W_{1}}$ have weighted degree 1, then $W_{1}$ and $W_{2}$ are equivalent. 
\end{theorem}

These are all results about B-model isomorphisms using the trivial group $\{\mathbf{0}\}$. The following is a result includes orbifolded B-models.

\begin{proposition}[Proposition 2.3.2 of \cite{FJJS11}]\label{bmodtensorprod} Suppose $W_{1}$ and $W_{2}$ are nondegenerate, quasihomogeneous polynomials with no variables in common. If $G_{1} \le \text{SL}(W_{1})$ and $G_{2} \le \text{SL}(W_{2})$, then $G_{1} \times G_{2}$ is contained in $\text{SL}(W_{1} + W_{2})$, $G_{1} \times G_{2}$ fixes $W_{1} + W_{2}$, and we have an isomorphism
\begin{align*}
\B[W_{1},G_{1}] \otimes \B[W_{2},G_{2}] \isom \B[W_{1} + W_{2}, G_{1} \times G_{2}].
\end{align*}
\end{proposition}

Note that \thmref{webbs_thm} is a type of Group-Weights result on the B-side. However, Group-Weights does not hold in general for B-models as the next example demonstrates. 

\begin{example}[Example 5.1.4 of \cite{Suggs}] \label{bmodelcondefcounterex} Let $W_{1} = x^{4} + y^{4}$ and $W_{2} = x^{3}y + xy^{3}$. Both polynomials have weights $\left(\frac{1}{4}, \frac{1}{4} \right)$. The set $\{1, y, y^{2}, x, xy, xy^{2}, x^{2}, x^{2}y, x^{2}y^{2} \}$ is a basis for both $\mathcal{Q}_{W_{1}}$ and $\mathcal{Q}_{W_{2}}$. One can verify that any ring homomorphism from $\mathcal{Q}_{W_{1}}$ to $\mathcal{Q}_{W_{2}}$ will not be surjective, so we see that $\B[W_{1}, \{\mathbf{0}\}] \not \isom \B[W_{2},\{\mathbf{0}\}]$. But notice that $x^{2}y^{2}$ has weighted degree 1. We see that any choice of basis for $\mathcal{Q}_{W_{1}}$ or $\mathcal{Q}_{W_{2}}$ will contain a monomial of weighted degree 1. Therefore this does not contradict Webb's Theorem. 
\end{example}

This shows that Group-Weights is not sufficient for B-model isomorphisms. This also shows that deformation invariance does not hold in general on the B-side, since there is no way to deform $x^{4} + y^{4}$ into $x^{3}y + xy^{3}$ while maintaining isomorphic Milnor rings.


\section{Isomorphism Extension Theorem}

Though the Group-Weights theorem does not hold in general for B-models, we still want to find instances where it does. So given equivalent singularities $W_{1}$, $W_{2}$ with a common group $G \le \text{SL}(n, \C)$ that fixes them both, we want to find cases when their corresponding B-models $\B[W_{1}, G]$ and $\B[W_{2}, G]$ are also isomorphic. We will need to impose a condition on our polynomials and groups, which condition in part stems from the following definition.


\begin{definition}[Property (*) of \cite{FJJS11}]\label{propertystar}
Let $W$ be a nondegenerate, invertible polynomial, and let $G$ be an admissible group of symmetries of $W$. The pair $(W,G)$ has \emph{Property (*)} if
\begin{enumerate}
\item $W$ can be decomposed as $W = \sum_{i = 1}^{M} W_{i}$, where the $W_{i}$ are themselves invertible polynomials having no variables in common with any other $W_{j}$. 

\item For any element $g$ of $G$ whose associated sector $\A_{g} \subseteq \A[W,G]$ is nonempty, and for each $i \in \{1, \dots, M\}$ the action of $g$ fixes either all of the variables in $W_{i}$ or none of them.

\item For any element $g'$ of $G^{T}$ whose associated sector of $\B_{g'} \subseteq \B[W^{T}, G^{T}]$ is nonempty, and for each $i \in \{1, \dots, M\}$ the action of $g'$ fixes either all of the variables in $W^{T}_{i}$ or none of them.

\end{enumerate}
\end{definition}

Here the \emph{sector} of an A- or B-model corresponding to a group element $g$ refers to the subset of the vector space basis containing the elements of the form $\lfloor m; g \rceil$.

Property (*) in \cite{FJJS11} is a generalization of the \emph{well behaved} condition for a polynomial/group pair $(W,G)$ given in \defref{wellbehaved}. We note that for the following polynomials, any possible choice of group (that fixes the polynomial and is contained in SL$(n, \C)$) will form a well-behaved pair: fermats, loops in any number of variables, and any admissible polynomial in two variables. We can further admit arbitrary sums of fermat and loop polynomials in distinct variables, together with any of their symmetry groups  (see Remark 1.1.1 of \cite{FJJS11}).










\begin{theorem}\label{isomextensionthm} Let $W_{1}$ and $W_{2}$ be admissible polynomials with $\phi: \mathcal{Q}_{W_{1}} \rightarrow \mathcal{Q}_{W_{2}}$ an equivariant isomorphism of graded Frobenius algebras, and let $G$ be a group that preserves both $W_{1}$ and $W_{2}$. If $(W_{1},G)$ and $(W_{2},G)$ are well behaved, then $\phi$ extends to an isomorphism $\psi: \B[W_{1}, G] \rightarrow \B[W_{2},G]$.
\end{theorem}

Consider the following diagram:

\begin{displaymath}
\xymatrix{
\B[W_{1}, G] \ar[r]^{\psi}  & \B[W_{2},G]  \\
\B[W_{1}, \{\mathbf{0}\}] \ar[r]^{\phi} \ar@{==>}[u] & \B[W_{2},\{\mathbf{0}\} ]  \ar@{==>}[u]\\
}
\end{displaymath}

The bottom horizontal arrow is the isomorphism we are given by hypothesis. The dashed vertical arrows represent an orbifolding (and, generally speaking, there won't exist an isomorphism going from bottom to top). The top horizontal arrow is the map that is conjectured to exist. In essence, we want to take the map $\phi$ that we are given, and use it to create an isomorphism of orbifolded B-models.

\begin{proof} By hypothesis, there exists an equivariant isomorphism $\phi: \mathcal{Q}_{W_{1}} \rightarrow \mathcal{Q}_{W_{2}}$. Also by hypothesis, we'll assume that $\phi$ is equivariant with respect to $G$. Suppose that a monomial basis for $\mathcal{Q}_{W_{1}}$ is $\text{span}_{\C} \{m_{1} = 1, \dots, m_{k} \}$. We obtain a basis for $\mathcal{Q}_{W_{2}}$ with $\text{span}_{\C} \{\phi(m_{1}) = 1, \dots, \phi(m_{k}) \}$. 

Suppose that $\left(\mathcal{Q}_{W_{1}}  \right)^{G} = \text{span}_{\C}\{p_{1}, \dots, p_{l} \}$, where each $p_{i} = m_{j}$ for some $j$, and $l \le k$. Since $\phi$ is equivariant, we have that $g \cdot \phi(p_{i}) = \phi(g \cdot p_{i}) = \phi(p_{i})$. Therefore $\text{span}_{\C}\{\phi(p_{1}), \dots, \phi(p_{i}) \} \subseteq \left(\mathcal{Q}_{W_{2}}  \right)^{G}$. But if we take an $m_{i}$ not preserved under the action of $G$, we get $g \cdot \phi(m_{i}) = \phi(g \cdot m_{i}) = \phi(cm_{i}) = c\phi(m_{i})$ for some constant $c \ne 1$. Therefore $\left(\mathcal{Q}_{W_{2}}  \right)^{G} = \text{span}_{\C}\{\phi(p_{1}), \dots, \phi(p_{i}) \}$.

Notice that the same process works even if we first restrict $W_{1}$ to a fixed locus of a group element. So for $\left(\mathcal{Q}_{W_{1}\mid_{\text{fix}(g)} }  \right)^{G}$, we can write it as $\text{span}_{\C}\{r_{i}\}$ where the $r_{i}$ form a subset of the $m_{i}$. We see that $\left(\mathcal{Q}_{W_{2}\mid_{\text{fix}(g)} }  \right)^{G} = \text{span}_{\C}\{\phi(r_{i})\}$ as before. This gives us the following:  there are (not necessarily distinct) group elements $h_{1}, \dots, h_{l}$ such that 
\begin{align*}
\B[W_{1},G] &= \text{span}_{\C}\{\lfloor p_{1}; h_{1} \rceil, \dots, \lfloor p_{l}; h_{l} \rceil   \}, \\
\B[W_{2},G] &= \text{span}_{\C}\{\lfloor \phi(p_{1}); h_{1} \rceil, \dots, \lfloor \phi(p_{l}); h_{l} \rceil   \}.
\end{align*}

Now define the map $\psi: \B[W_{1},G] \rightarrow \B[W_{2},G]$ by $\psi(\lfloor p_{i}; h_{i} \rceil ) = \lfloor \phi(p_{i}); h_{i} \rceil$. Notice that $\psi$ is a well-defined bijection that preserves the vector space grading. Also $\psi$ maps the identity $\lfloor 1; \mathbf{0} \rceil$ to the identity $\lfloor 1; \mathbf{0} \rceil$. 

That $\psi$ preserves the pairing is also easy to show. Let $B_{1} = \B[W_{1},G]$ and $B_{2} = \B[W_{2},G]$. Using the properties of pairings, we have for $h_{i} + h_{j} = \mathbf{0}$, 
\begin{align*}
\langle \lfloor p_{i}; h_{i} \rceil, \lfloor p_{j}; h_{j} \rceil \rangle_{B_{1}} = \langle p_{i}, p_{j} \rangle_{\mathcal{Q}_{W_{1}}} = \langle \phi(p_{i}), \phi(p_{j}) \rangle_{\mathcal{Q}_{W_{2}}} = \langle \lfloor \phi(p_{i}); h_{i} \rceil, \lfloor \phi(p_{j}); h_{j} \rceil \rangle_{B_{2}}.
\end{align*}
Since all other pairings are zero, this shows that $\psi$ respects the pairing. 

Now for the products. For basis elements $\alpha, \beta$ of $B_{1}$, we want to show that $\psi(\alpha \star \beta) = \psi(\alpha) \star \psi(\beta)$. We'll consider the case where $\text{fix}(h_{i}) \cup \text{fix}(h_{j}) \cup \text{fix}(h_{i} + h_{j}) = \C^{n}$. Otherwise, both products will be zero. First,
\begin{align*}
\psi(\alpha \star \beta) = \psi(\lfloor p_{i}; h_{i} \rceil \star \lfloor p_{j}; h_{j} \rceil ) = \psi( \lfloor \gamma_{1} p_{i}p_{j}; h_{i} + h_{j} \rceil) = \lfloor \phi(\gamma_{1}p_{i}p_{j}); h_{i} + h_{j} \rceil = \lfloor \phi(\gamma_{1})\phi(p_{i}p_{j}); h_{i} + h_{j} \rceil. 
\end{align*}
The last equality comes from considering $\gamma_{1}$ as a monomial in $\mathcal{Q}_{W_{1}}$. Here we have
\begin{align*}
\gamma_{1} = \frac{\mu_{h_{i} \cap h_{j}} \text{Hess}(W_{1}|_{\text{fix}(h_{i} + h_{j})})  }{\mu_{h_{i} + h_{j}} \text{Hess}(W_{1}|_{\text{fix}(h_{i}) \cap \text{fix}(h_{j})}) }. 
\end{align*}

Second, we have
\begin{align*}
\psi(\alpha) \star \psi(\beta) = \lfloor \phi(p_{i}); h_{i} \rceil \star \lfloor \phi(p_{j}); h_{j} \rceil = \lfloor \gamma_{2}\phi(p_{i})\phi(p_{j}); h_{i} + h_{j} \rceil = \lfloor \gamma_{2}\phi(p_{i}p_{j}); h_{i} + h_{j} \rceil.
\end{align*}
Here we have
\begin{align*}
\gamma_{2} = \frac{\mu_{h_{i} \cap h_{j}} \text{Hess}(W_{2}|_{\text{fix}(h_{i} + h_{j})})  }{\mu_{h_{i} + h_{j}} \text{Hess}(W_{2}|_{\text{fix}(h_{i}) \cap \text{fix}(h_{j})}) }. 
\end{align*}
Previously, we computed bases for the Milnor rings of $W_{1}$ and $W_{2}$ after restricting to fixed loci and taking $G$-invariants. Since the dimension remained the same between $W_{1}$ and $W_{2}$ after these operations, we see that $\mu_{h_{i} \cap h_{j}}$ for $W_{1}$ equals $\mu_{h_{i} \cap h_{j}}$ for $W_{2}$ and similarly for $\mu_{h_{i} + h_{j}}$. So it just remains to check how $\phi$ deals with the respective Hessians. That is, we will have $\lfloor \phi(\gamma_{1})\phi(p_{i}p_{j}); h_{i} + h_{j} \rceil = \lfloor \gamma_{2}\phi(p_{i}p_{j}); h_{i} + h_{j} \rceil$ if we can show $\phi(\gamma_{1}) = \gamma_{2}$. We'll consider the behavior of group elements, and break this down into cases.  

\emph{Case 1}:  $h_{i} = h_{j} = \mathbf{0}$. Notice that $W_{i}$ restricted to the fixed locus is just $W_{i}$ again. So the Hessians divide each other, which shows that $\gamma_{1} = \gamma_{2}$. Further, $\mu_{h_{i} \cap h_{j}} = \mu_{h_{i} + h_{j}}$, which shows that $\gamma_{1} = \gamma_{2} = 1$. Therefore $\phi(\gamma_{1}) = \gamma_{2}$. 

\emph{Case 2}:  one of $h_{i}, h_{j} = \mathbf{0}$. Without loss of generality, $h_{i} = \mathbf{0}$. So $\gamma_{1} = \dfrac{\mu_{h_{j}} \text{Hess}(W_{1}|_{\text{fix}(h_{j})})  }{\mu_{h_{j}} \text{Hess}(W_{1}|_{\text{fix}(h_{j})}) } = 1$. Similarly, $\gamma_{2} = 1$. Therefore $\phi(\gamma_{1}) = \gamma_{2}$. 

\emph{Case 3}:  Both $h_{i}, h_{j}$ are nonzero. By hypothesis on the behavior of our group elements,we will have the fixed locus of $h_{i}$ and $h_{j}$ trivial. But $h_{i} + h_{j}$ must be $\mathbf{0}$ in order to get a nonzero product. Therefore $\gamma_{1} = \dfrac{\text{Hess}(W_{1})}{\mu}$, $\gamma_{2} = \dfrac{\text{Hess}(W_{2})}{\mu}$. We will have $\phi(\gamma_{1}) = \gamma_{2}$ if we can show that $\phi(\text{Hess}(W_{1})) = \text{Hess}(W_{2})$.

\begin{lemma}\label{hessianstohessians} If $\phi: \B[W_{1},\{\mathbf{0}\}] \rightarrow \B[W_{2},\{\mathbf{0}\}]$ is an isomorphism of B-models, then $\phi(\text{Hess}(W_{1})) = \text{Hess}(W_{2})$.
\end{lemma}

\begin{proof} Let $B_{1} = \B[W_{1},\{\mathbf{0}\}]$ and $B_{2} = \B[W_{2},\{\mathbf{0}\}]$. Suppose $m_{1}, m_{2}$ are monomials in the basis of $B_{1}$ such that $m_{1}m_{2}$ spans the sector of highest degree in $B_{1}$. Since $\phi$ is an isomorphism, we can write $\B_{2} = \text{span}_{\C} \{\phi(m) \mid m \text{ is a basis element of } B_{1} \}$. Also, we know that $\phi$ preserves pairings:
\begin{align*}
\langle m_{1}, m_{2} \rangle_{B_{1}} = \langle \phi(m_{1}), \phi(m_{2}) \rangle_{B_{2}}. 
\end{align*}
Recall that $m_{1}m_{2} = \dfrac{\langle m_{1}, m_{2} \rangle_{B_{1}}}{\mu}\text{Hess}(W_{1})$, where $\mu = \dim(B_{1})$. Since $B_{1} \isom B_{2}$, we also have that $\mu = \dim(B_{2})$. Now note that $\text{Hess}(W_{1}) = \dfrac{\mu(m_{1}m_{2})}{\langle m_{1}, m_{2}\rangle_{B_{1}}}$. Apply $\phi$:
\begin{align*}
\phi(\text{Hess}(W_{1})) = \phi\left(\frac{\mu(m_{1}m_{2})}{\langle m_{1}, m_{2}\rangle_{B_{1}}}    \right) = \frac{\mu\phi(m_{1}m_{2})}{\langle m_{1}, m_{2}\rangle_{B_{1}}} = \frac{\mu\phi(m_{1}m_{2})}{\langle \phi(m_{1}), \phi(m_{2})\rangle_{B_{2}}}.
\end{align*}
On the other hand, we know by the isomorphism that the element $\phi(m_{1}m_{2}) = \phi(m_{1})\phi(m_{2})$ spans the sector of highest degree in $B_{2}$. We have that $\phi(m_{1})\phi(m_{2}) = \dfrac{\langle \phi(m_{1}), \phi(m_{2}) \rangle_{B_{2}}}{\mu}\text{Hess}(W_{2})$. So then
\begin{align*}
\text{Hess}(W_{2}) = \frac{\mu\phi(m_{1})\phi(m_{2})}{\langle \phi(m_{1}), \phi(m_{2})\rangle_{B_{2}}} = \frac{\mu\phi(m_{1}m_{2})}{\langle \phi(m_{1}), \phi(m_{2})\rangle_{B_{2}}}.
\end{align*}
This shows that $\phi(\text{Hess}(W_{1})) = \text{Hess}(W_{2})$, as desired.
\end{proof}

Back to the theorem now, we have by  \lemref{hessianstohessians} the result we were seeking. So this verifies Case 3. And, we notice, that this is enough to prove the theorem.
\end{proof}

We can now generalize the result to sums of polynomials.

\begin{corollary}\label{isomextensioncorollary} Let $W = W_{1} + W_{2}$ and $V = V_{1} + V_{2}$ be sums of admissible polynomials in distinct variables where $\phi_{i}: \mathcal{Q}_{W_{i}} \rightarrow \mathcal{Q}_{V_{i}}$ is an equivariant isomorphism of graded Frobenius algebras for each $i$. If $(W_{i},G_{i})$ and $(V_{i},G_{i})$ form well-behaved pairs for each $i$, then there exists an isomorphism $\psi: \B[W, G] \rightarrow \B[V,G]$ where $G = G_{1} \times G_{2}$.
\end{corollary}

\begin{proof} First we'll construct an isomorphism $\phi: \B[W, \{\mathbf{0}\}] \rightarrow \B[V, \{\mathbf{0}\}]$ using the $\phi_{i}$.

\emph{Claim}:  By the tensor product structure (see \propref{bmodtensorprod}), we know that any monomial $m_{i}$ in the basis of $\mathcal{Q}_{W}$ can be written as $\alpha_{i} \beta_{i}$ where the $\alpha_{i}$ is in the basis of $\mathcal{Q}_{W_{1}}$ and the $\beta_{i}$ is in the basis of $\mathcal{Q}_{W_{2}}$. We can define $\phi$ by $\phi: m_{i} \mapsto \phi_{1}(\alpha_{i})\phi_{2}(\beta_{i})$ and extend linearly.

\emph{Proof of Claim}: It is easy to verify that $\phi$ is a bijection, is linear, sends the identity to the identity, and preserves degrees. To show that $\phi$ respects the pairing, we note that
\begin{align*}
\langle \phi(m_{i}), \phi(m_{j}) \rangle_{\mathcal{Q}_{V}} &= \langle \phi_{1}(\alpha_{i})\phi_{2}(\beta_{i}), \phi_{1}(\alpha_{j})\phi_{2}(\beta_{j}) \rangle_{\mathcal{Q}_{V}}\\
&= \langle \phi_{1}(\alpha_{i}), \phi_{1}(\alpha_{j}) \rangle_{\mathcal{Q}_{V_{1}}} \langle \phi_{2}(\beta_{i}), \phi_{2}(\beta_{j}) \rangle_{\mathcal{Q}_{V_{2}}} \\
&= \langle \alpha_{i}, \alpha_{j} \rangle_{\mathcal{Q}_{W_{1}}} \langle \beta_{i}, \beta_{j} \rangle_{\mathcal{Q}_{W_{2}}} \\
&= \langle \alpha_{i}\beta_{i}, \alpha_{j}
\beta_{j}\rangle_{\mathcal{Q}_{W}} \\
&= \langle m_{i}, m_{j} \rangle_{\mathcal{Q}_{W}}.
\end{align*}
For the products, we note that 
\begin{align*}
\phi(m_{i}m_{j}) &= \phi(\alpha_{i}\beta_{i} \alpha_{j}\beta_{j}) = \phi(\alpha_{i}\alpha_{j}\beta_{i}\beta_{j}) = \phi_{1}(\alpha_{i}\alpha_{j})\phi_{2}(\beta_{i}\beta_{j}) = \phi_{1}(\alpha_{i})\phi_{1}(\alpha_{j})\phi_{2}(\alpha_{i})\phi_{2}(\alpha_{j}) \\ 
&= \phi_{1}(\alpha_{i})\phi_{2}(\beta_{i}) \phi_{1}(\alpha_{j})\phi_{2}(\beta_{j}) = \phi(\alpha_{i}\beta_{i}) \phi(\alpha_{j}\beta_{j}) = \phi(m_{i})\phi(m_{j}).
\end{align*}
Therefore $\phi$ really is an isomorphism of graded Frobenius algebras. We further check that $\phi$ is equivariant: for $g \in G$, we have $g \cdot \phi(m) = g \cdot (\phi_{1}(\alpha)\phi_{2}(\beta)) = (g \cdot \phi_{1}(\alpha))(g \cdot \phi_{2}(\beta))$, since $\alpha$ and $\beta$ are in distinct variables, $= \phi_{1}(g \cdot \alpha) \phi_{2}(g \cdot \beta)$, since $\phi_{1}$ and $\phi_{2}$ are equivariant, $= \phi(g \cdot m)$. 

Now given our map $\phi$, we see that $W$ and $V$ are equivalent singularities. Construct map $\psi$ as before, but with using $\phi$ as the base map. The only thing left to check is that $\psi$ respects products for group elements with nontrivial fixed locus. First note that with the $W_{i}$ in distinct variables, the block matrix structure of the second parital derivatives of $W$ will give us Hess$(W) = \text{Hess}(W_{1})\text{Hess}(W_{2})$. It follows that $\phi$ sends Hess$(W_{i})$ to Hess$(V_{i})$ by \lemref{hessianstohessians} and by construction. Now the group elements $g, h$ have to fix all the variables in either $W_{1}$ or $W_{2}$ by the hypothesis of the symmetry group structure. This way any quotient of Hessians will reduce to either Hess$(W_{1})$ or Hess$(W_{2})$. This shows that $\psi$ respects the products, and gives us the desired isomorphism.
\end{proof}





We now include a brief result on equivariant isomorphisms.

\begin{lemma}\label{equivariantisoms}
Suppose $(W,G)$ and $(V,G)$ are well behaved. Then an isomorphism $\phi: \mathcal{Q}_{W} \rightarrow \mathcal{Q}_{V}$ is equivariant if and only if we have equivariant isomorphisms $\phi_{i}: \mathcal{Q}_{W_{i}} \rightarrow \mathcal{Q}_{V_{i}}$ for each $i$. 
\end{lemma}

\begin{proof}
$(\Rightarrow)$ Suppose that $\phi: \mathcal{Q}_{W} \rightarrow \mathcal{Q}_{V}$ is an equivariant isomorphism of graded Frobenius algebras. We can write $W = W_{1} + \dots + W_{n}$ and $V = V_{1} + \dots + V_{n}$ where each $W_{i}$ is in the same variables as $V_{i}$ but $W_{i}$ is in distinct variables from $W_{j}$ for all $i \ne j$. We can also write $G = G_{1} \times \dots \times G_{n}$, where $G_{i}$ preserves either all or none of the variables of $W_{i}, V_{i}$ for each $i$. By \propref{bmodtensorprod}, we can consider $\mathcal{Q}_{W} \isom \mathcal{Q}_{W_{1}} \otimes \dots \otimes \mathcal{Q}_{W_{n}}$ and $\mathcal{Q}_{V} \isom \mathcal{Q}_{V_{1}} \otimes \dots \otimes \mathcal{Q}_{V_{n}}$. From the tensor product structure, we find that there exists a basis of each $\mathcal{Q}_{W_{i}}$ that is a subset of a basis of $\mathcal{Q}_{W}$. By restricting $\phi$ to the variables of $W_{i}$, we obtain an equivariant isomorphism $\phi_{i}: \mathcal{Q}_{W_{i}} \rightarrow \mathcal{Q}_{V_{i}}$ for each $i$. 

$(\Leftarrow)$ Conversely, suppose that we have equivariant isomorphisms $\phi_{i}: \mathcal{Q}_{W_{i}} \rightarrow \mathcal{Q}_{V_{i}}$ for each $i$. The argument in the proof of \corref{isomextensioncorollary} shows how to construct an equivariant isomorphism $\phi: \mathcal{Q}_{W} \rightarrow \mathcal{Q}_{V}$ in the case that $n = 2$. Extending by induction gives us the result for all $n$. 
\end{proof}

We are now ready to obtain the main result of the paper.

\begin{theorem}\label{mainresult}
Let $(W,G)$ and $(V, G)$ be well behaved. If $\phi: \mathcal{Q}_{W} \rightarrow \mathcal{Q}_{V}$ is an equivariant isomorphism of graded Frobenius algebras, then $\phi$ extends to an isomorphism $\psi: \B[W, G] \rightarrow \B[V, G]$.
\end{theorem}

\begin{proof}

Given $\phi: \mathcal{Q}_{W} \rightarrow \mathcal{Q}_{V}$ an equivariant isomorphism of graded Frobenius algebras, we can apply \lemref{equivariantisoms} to obtain $\phi_{i}: \mathcal{Q}_{W_{i}} \rightarrow \mathcal{Q}_{V_{i}}$ that are also equivariant isomorphisms of graded Frobenius algebras. We can then extend \corref{isomextensioncorollary} by induction in the case that $W = W_{1} + \dots + W_{n}$ and $V = V_{1} + \dots + V_{n}$ are sums of admissible polynomials in distinct variables such that each $W_{i}$ is singularity equivalent to $V_{i}$, and $G_{i}$ is a group that preserves both $W_{i}$ and $V_{i}$ for each $i$ such that each group element of $G_{i}$ fixes either all or none of the variables of $W_{i}$ and $V_{i}$.
\end{proof}

\thmref{mainresult} actually applies to a large class of isomorphisms. For example, any diagonal isomorphism is equivariant.

\begin{definition}
Suppose $\phi: B_{1} \rightarrow B_{2}$ is an isomorphism of B-models. Say that $B_{1}$ has basis $\{a_{1}, \dots, a_{n} \}$ and $B_{2}$  has basis $\{b_{1}, \dots, b_{n}\}$. We say that $\phi$ is \emph{diagonal} if we can write $\phi(a_{i}) = c_{i}b_{i}$ for $c_{i} \in \C$ nonzero (possibly after reordering the basis elements).
\end{definition}

\begin{theorem} Any diagonal isomorphism of Landau-Ginzburg B-models is equivariant. 
\end{theorem}

\begin{proof} Suppose $\phi: B_{1} \rightarrow B_{2}$ is a diagonal isomorphism of B-models. Let $B_{1}$ have basis $\{a_{1}, \dots, a_{n} \}$ and $B_{2}$  have basis $\{b_{1}, \dots, b_{n}\}$. Write $\phi(a_{i}) = c_{i}b_{i}$ for $c_{i} \in \C$ nonzero (reordering if necessary). Now notice the following. For any $g \in G$, 
\begin{align*}
\phi(g \cdot a_{i}) &= \phi(\det(g) a_{i} \circ g) = \det(g) \phi(a_{i} \circ g) = \det(g) c_{i}(b_{i} \circ g). \\
g \cdot \phi(a_{i}) &= g \cdot c_{i}b_{i} = \det(g) c_{i} (b_{i} \circ g). 
\end{align*}
This happens since $a_{i} \circ g$ is just a constant times $a_{i}$. Because $\phi(g \cdot a_{i}) = g \cdot \phi(a_{i})$ for each $i$, we see that $\phi$ is equivariant.
\end{proof}

\section{Examples}

In the following examples, we will demonstrate how we can apply these results.

\begin{example}[see Theorems 6.3 and 6.6 of \cite{Cor}]\label{examplefirst}
We can compute for all $n \ge 2$, 
\begin{displaymath}
\xymatrix{
\B[x^{2} + y^{2n}, \{\mathbf{0}\}]  \ar@{<->}[r] &  \B[x^{2} + xy^{n} + y^{2n}, \{\mathbf{0}\}] & \B[x^{2} + xy^{n}, \{\mathbf{0}\}] \ar@{<->}[l] }
\end{displaymath}
Label $B_{1} = \B[x^{2} + y^{2n}, \{\mathbf{0}\}]$, $B_{2} = \B[x^{2} + xy^{n} + y^{2n}, \{\mathbf{0}\}]$, and $B_{3} = \B[x^{2} + xy^{n}, \{\mathbf{0}\}]$. Each unorbifolded B-model has basis $\text{span}_{\C}\{1, y, \dots, y^{2n-2}\}$. We can define a map $\phi_{1}: B_{1} \rightarrow B_{3}$ by $\phi_{1}(y^{a}) = c^{a}y^{a}$, where $c$ is a complex number that satisfies $c^{2n-2} = \frac{3}{4}$. We can also define a map $\phi_{2}: B_{2} \rightarrow B_{3}$ by $\phi_{2}(y^{a}) = c^{a}y^{a}$, where $c$ is a complex number that satisfies $c^{2n-2} = -3$. One can verify that $\phi_{1}$ and $\phi_{2}$ are isomorphisms of graded Frobenius algebras \cite{Cor}. And, since these are diagonal maps, they are equivariant.

If $n$ is odd, then $G = \left\langle \left( \frac{1}{2}, \frac{1}{2} \right) \right\rangle$ fixes each polynomial. By \thmref{isomextensionthm}, we have for all odd $n > 2$
\begin{displaymath}
\xymatrix{
\B[x^{2} + y^{2n}, G]  \ar@{<->}[r] \ar@{<==}[d] &  \B[x^{2} + xy^{n} + y^{2n}, G] \ar@{<==}[d] & \B[x^{2} + xy^{n}, G] \ar@{<->}[l] \ar@{<==}[d] \\
\B[x^{2} + y^{2n}, \{\mathbf{0}\}]  \ar@{<->}[r] &  \B[x^{2} + xy^{n} + y^{2n}, \{\mathbf{0}\}] & \B[x^{2} + xy^{n}, \{\mathbf{0}\}] \ar@{<->}[l]  }
\end{displaymath}

Applying mirror symmetry to B-models built with invertible polynomials, we get the following mirror diagram. 
\begin{displaymath}
\xymatrix{
\A[x^{2} + y^{2n}, \langle (\frac{1}{2},0), (0, \frac{1}{2n}) \rangle]  \ar@{<->}[r] \ar@{==>}[d] & \A[x^{2}y + y^{n}, \langle (- \frac{1}{2n}, \frac{1}{n}) \rangle] \ar@{==>}[d] \\
\A[x^{2} + y^{2n}, \langle (\frac{1}{2}, \frac{1}{2n}) \rangle]  \ar@{<->}[r]  & \A[x^{2}y + y^{n}, \langle ( \frac{n-1}{2n}, \frac{1}{n}) \rangle]   }
\end{displaymath}

Here the unorbifolded B-models in the previous diagram correspond to the top row of A-models in the above diagram. The orbifolded B-models of the previous diagram correspond to the A-models on the bottom row in the above diagram. Notice that the isomorphism $\A[x^{2} + y^{2n}, \langle (\frac{1}{2}, \frac{1}{2n}) \rangle] \isom \A[x^{2}y + y^{n}, \langle ( \frac{n-1}{2n}, \frac{1}{n}) \rangle]$ is the result of the B-model isomorphisms we just computed together with mirror symmetry. Further note that the groups used for these A-models are distinct. Therefore, this is a  new isomorphism of A-models that does not stem from the Group-Weights theorem. Hence \thmref{isomextensionthm} tells us not only about isomorphisms of B-models, but can also be used to find new isomorphisms between A-models.


%

\end{example}

\begin{example}

Singularity theory suggests that adding \emph{quadratic forms} in distinct variables to polynomials will do nothing to affect the type of singularity defined. However, it is not immediately clear how the orbifolded Milnor ring structure of such an augmented polynomial will be affected.

Building off of \exampleref{examplefirst}, for each odd integer $n > 2$ let $W^{(1)}_{n} = x^{2} + y^{2n}$, $W^{(2)}_{n} = x^{2} + xy^{n} + y^{2n}$, and $W^{(3)}_{n} = x^{2} + xy^{n}$. Consider also the polynomial $V = z^{2} + w^{2}$. We already know that $\mathcal{Q}_{W^{(1)}_{n}}$, $\mathcal{Q}_{W^{(2)}_{n}}$, and $\mathcal{Q}_{W^{(3)}_{n}}$ are isomorphic to each other under equivariant maps. We also immediately see that $\mathcal{Q}_{V}$ is isomorphic to $\mathcal{Q}_{V}$ under the identity map, which is equivariant.

The group $G = \langle (\frac{1}{2}, \frac{1}{2}) \rangle$ preserves each of $W^{(1)}_{n}$, $W^{(2)}_{n}$, $W^{(3)}_{n}$, and $V$. Note also that the pair $(V,G)$ is well behaved. Let $G_{1} = \{(0,0,0,0)\}$, $G_{2} = G \times \{(0,0)\} = \langle (\frac{1}{2},\frac{1}{2}, 0,0)\rangle$, $G_{3} = \{(0,0)\} \times G = \langle (0,0,\frac{1}{2},\frac{1}{2})\rangle$, and $G_{4} = G \times G = \langle (\frac{1}{2},\frac{1}{2},0,0), (0,0,\frac{1}{2},\frac{1}{2})\rangle$. By \corref{isomextensioncorollary}, we have following B-model isomorphisms:
\begin{align*}
\B[W^{(1)}_{n} + V, G_{1}] &\isom \B[W^{(2)}_{n} + V, G_{1}] \isom \B[W^{(3)}_{n} + V, G_{1}], \\
\B[W^{(1)}_{n} + V, G_{2}] &\isom \B[W^{(2)}_{n} + V, G_{2}] \isom \B[W^{(3)}_{n} + V, G_{2}], \\
\B[W^{(1)}_{n} + V, G_{3}] &\isom \B[W^{(2)}_{n} + V, G_{3}] \isom \B[W^{(3)}_{n} + V, G_{3}], \\
\B[W^{(1)}_{n} + V, G_{4}] &\isom \B[W^{(2)}_{n} + V, G_{4}] \isom \B[W^{(3)}_{n} + V, G_{4}].
\end{align*}

We see that in these cases, equivalent singularities still yield the same orbifolded Milnor ring structure. We can take this one step further. If $(W,G)$ and $(V,G)$ are any well-behaved pairs where $W$ is singularity equivalent to $V$, and $(U,H)$ is a well-behaved pair where $U$ is a quadratic form in distinct variables from $W$, $V$ and $H$ is some orbifold group for $U$, we can apply \corref{isomextensioncorollary} to find that $\B[W + U, G \times H] \isom \B[V + U, G \times H]$.

\end{example}

%
%


\section{References}

\begingroup
\renewcommand{\section}[2]{}%

\endgroup


\begin{thebibliography}{99}
\bibitem{ABS11}
  Michela Artebani, Samuel Boissi\`{e}re, and Alessandra Sarti,
  \textit{The Berglund-H\"{u}bsch-Chiodo-Ruan mirror symmetry for K3 surfaces}, Journal de Math\'ematiques Pures et Appliqu\'ees, vol. 102, no. 4 (2014), 758–781.
  
\bibitem{Hen}
P. Berglund and M. Henningson, \textit{Landau-Ginzburg orbifolds, mirror symmetry and the elliptic genus}, Nucl. Phys. B, 433(2):311–332, 1995.  
  
\bibitem{BH}
P. Berglund and T. H\"ubsch, \textit{A generalized construction of mirror manifolds}, Nucl. Phys. B 393 (1993)
377–391.  

\bibitem{Cor} N. Cordner, \emph{Isomorphisms of Landau-Ginzburg B-Models}, Master's thesis, Brigham Young University, June 2016. 

  
\bibitem{FJR07}
Huijun Fan, Tyler J. Jarvis, and Yongbin Ruan, \textit{The Witten equation, mirror symmetry and quantum
singularity theory}, Annals of Mathematics, vol. 178, no. 1 (2013), 1–106.   
  
\bibitem{FJJS11}
  Amanda Francis, Tyler Jarvis, Drew Johnson, and Rachel Suggs, \textit{Landau-Ginzburg mirror symmetry for orbifolded Frobenius algebras}, Proceedings of Symposia in Pure Mathematics, vol. 85 (2012), 333–353. 
 

\bibitem{IV}
K. Intriligator and C. Vafa, \textit{Landau-Ginzburg orbifolds}, Nuclear Phys. B 339 (1990), no. 1, 95–120.

\bibitem{Ka1}
R. Kaufmann, \textit{Singularities with symmetries, orbifold Frobenius algebras and mirror symmetry}, Cont.
Math. 403, 67-116.

\bibitem{Ka2}
R. Kaufmann, \textit{Orbifold Frobenius algebras, cobordisms and monodromies}, Orbifolds in mathematics and physics (Madison,
WI, 2001), 135–161, Contemp. Math., 310, Amer. Math. Soc., Providence, RI, 2002.

\bibitem{Ka3}
R. Kaufmann, \textit{Orbifolding Frobenius algebras}, Internat. J. Math. 14 (2003), no. 6, 573–617.

\bibitem{Kra09}
Marc Krawitz, \textit{FJRW rings and Landau-Ginzburg mirror symmetry}, PhD thesis, University of Michigan, 2010.

\bibitem{KS}
M. Kreuzer and H. Skarke, \textit{On the classification of quasihomogeneous functions}, Comm. Math.
Phys. 150 (1992), no. 1, 137–147. MR 1188500 (93k:32075).

\bibitem{Suggs}
Rachel Suggs. \textit{An explanation of strange duality in terms of Landau-Ginzburg mirror symmetry}. Honors thesis, Brigham Young University, August 2012.


\bibitem{Tay13}
Julian Tay, \textit{Poincar\'{e} polynomial of FJRW rings and the group-weights conjecture}, Master's thesis, Brigham Young University, May 2013.







\end{thebibliography}
\end{document}